\documentclass[11pt]{amsart}
\usepackage{mathrsfs,latexsym,amsfonts,amssymb}
\newtheorem{theorem}{Theorem}[section]
\newtheorem{lemma}[theorem]{Lemma}
\newtheorem{corollary}[theorem]{Corollary}
\newtheorem{question}[theorem]{Question}
\newtheorem{example}[theorem]{Example}
\theoremstyle{definition}
\newtheorem{definition}[theorem]{Definition}
\newtheorem{proposition}[theorem]{Proposition}
\theoremstyle{remark}
\newtheorem{remark}[theorem]{Remark}

\begin{document}

\title[Some generalized metric properties of $n$-semitopological groups]
{Some generalized metric properties of $n$-semitopological groups}

\author{Fucai Lin*}
\address{(Fucai Lin): 1. School of mathematics and statistics,
  Minnan Normal University, Zhangzhou 363000, P. R. China; 2. Fujian Key Laboratory of Granular Computing and Application, Minnan Normal University, Zhangzhou 363000, P. R. China}
\email{linfucai2008@aliyun.com; linfucai@mnnu.edu.cn}

\author{Xixi Qi}
\address{(Xixi Qi): 1. School of mathematics and statistics,
  Minnan Normal University, Zhangzhou 363000, P. R. China}
\email{2473148247@qq.com}

\thanks{The authors are supported by Fujian Provincial Natural Science Foundation of China (No: 2024J02022) and the NSFC (Nos. 11571158).\\
*corresponding author}


\begin{abstract}
A semitopological group $G$ is called {\it an $n$-semitopological group}, if for any $g\in G$ with $e\not\in\overline{\{g\}}$ there is a neighborhood $W$ of $e$ such that $g\not\in W^{n}$, where $n\in\mathbb{N}$. The class of $n$-semitopological groups ($n\geq 2$) contains the class of paratopological groups and Hausdorff quasi-topological groups. Fix any $n\in\mathbb{N}$. Some properties of $n$-semitopological groups are studied, and some questions about $n$-semitopological groups are posed. Some generalized metric properties of $n$-semitopological groups are discussed, which contains mainly results are that (1) each Hausdorff first-countable 2-semitopological group admits a coarser semi-metrizable topology; (2) each locally compact, Baire and $\sigma$-compact 2-semitopological group is a topological group; (3) the condensation of some kind of 2-semitopological groups topologies are given. Finally, some cardinal invariants of $n$-semitopological groups are discussed.
\end{abstract}

\maketitle

\section{Introduction and Preliminaries}
Let $G$ be a group, and let $\mathscr{F}$ be a topology on $G$. We say that

$\bullet$ $G$ is a {\it semitopological group} if the product map of $G\times G$ into $G$ is separately
continuous under the topology $\mathscr{F}$;

$\bullet$ $G$ is a {\it quasitopological group} if, under the topology $\mathscr{F}$, the space $G$ is a semitopological group and the inverse map of $G$ onto itself
associating $x^{-1}$ with arbitrary $x\in G$ is continuous;

$\bullet$ $G$ is a {\it paratopological group} if
the product map of $G \times G$ into $G$ is jointly continuous under the topology $\mathscr{F}$;

$\bullet$ $G$ is a {\it topological group} if, under the topology $\mathscr{F}$, the space $G$ is a paratopological group and the inverse map of $G$ onto itself
associating $x^{-1}$ with arbitrary $x\in G$ is continuous.

The classes of semitopological groups, quasitopological groups, paratopological groups and topological groups were studied from twentieth century, see \cite{AA}.
In \cite{ER1}, R. Ellis proved that each locally compact Hausdorff semitopological group is a topological group, which shows that each compact Hausdorff semitopological group is a topological group. Recently, the concept of almost paratopological group has been introduced by E. Reznichenko in \cite{RE}, which is a generalized of paratopological groups and Hausdorff quasitopological groups. A semitopological group $G$ is called {\it almost paratopological}, if for any $g\in G$ with $e\not\in\overline{\{g\}}$ there is a neighborhood $W$ of $e$ such that $g\not\in W^{2}$. By applying the concept of almost paratopological group, it is proved in \cite{RE} that each compact almost paratopological group is a topological group. However, there exists a compact $T_{1}$ quasitopological group which is not a topological group, such as, the integer group with the finite complementary topology. In this paper, we define the following concept of $n$-semitopological group ($n\in\mathbb{N}$) and $\infty$-semitopological group, where each almost paratopological group is called $2$-semitopological group.

\begin{definition}
Fix an $n\in\mathbb{N}$. A semitopological group $G$ is called {\it an $n$-semitopological group}, if for any $g\in G$ with $e\not\in\overline{\{g\}}$ there is a neighborhood $W$ of $e$ such that $g\not\in W^{n}$.
In particular, $G$ is called {\it an $\infty$-semitopological group}, if for any $n\in\mathbb{N}$ and $g\in G$ with $e\not\in\overline{\{g\}}$ there is a neighborhood $W$ of $e$ such that $g\not\in W^{n}$.
\end{definition}

\begin{remark}
Clearly, each semitopological group and each almost paratopological group are just an $1$-semitopological group and a $2$-semitopological group respectively. In \cite{RE}, E. Reznichenko proved that all paratopological groups and Hausdorff quasi-topological groups are $\infty$-semitopological groups and 2-semitopological groups respectively. Obviously, there exists an $\infty$-semitopological group which is neither a paratopological group nor a quasi-topological group, see the following example.
\end{remark}

\begin{example}
Let $G=\mathbb{R}$ with the usual addition. Put $\mathscr{P}=\{[0, \frac{1}{n})-\mathbb{Q}_{+}: n\in\mathbb{N}\}$ and $\mathscr{B}=\{x+P: P\in\mathscr{P}, x\in \mathbb{R}\}$, where $\mathbb{Q}_{+}$ is the set all positive rational numbers. Let $\tau$ be a topology on $G$ such that $\mathscr{B}$ is a base for $\tau$. It is easy to see that $(G, \tau)$ is a Hausdorff $\infty$-semitopological group which is neither a paratopological group nor a quasi-topological group.
\end{example}

\begin{proof}
For any $k\in\mathbb{N}$, put $W_{k}=[0, \frac{1}{k})-\mathbb{Q}_{+}$. We first claim that $(G, \tau)$ is a Hausdorff $\infty$-semitopological group. Indeed, it is obvious that $(G, \tau)$ is Hausdorff. Fix any $n\in\mathbb{N}$. Take any $g\in G$ with $0\not\in\overline{\{g\}}=\{g\}$. Then $g\neq 0$, hence there exists $m\in\mathbb{N}$ such that $|g|>\frac{n}{m}$. Clearly, we have $g\not\in nW_{m}$ since $|g|>\frac{n}{m}$. Therefore, $(G, \tau)$ is a Hausdorff $\infty$-semitopological group. Now it suffices to prove that $(G, \tau)$ is neither a paratopological group nor a quasi-topological group. Clearly, $(G, \tau)$ is not a quasi-topological group since the neutral element $0$ does not have the symmetric neighborhood base. Moreover, for any $k\in\mathbb{N}$, since the numbers $\frac{1}{2k}-\frac{1}{2\sqrt{2}k}, \frac{1}{2\sqrt{2}k}$ belong to $W_{k}$, the set $W_{k}+W_{k}$ contains the rational number $\frac{1}{2k}-\frac{1}{2\sqrt{2}k}+\frac{1}{2\sqrt{2}k}=\frac{1}{2k}$. Therefore, $(G, \tau)$ is not a paratopological group.
\end{proof}

\begin{remark}
(1) If a $T_{0}$-quasitopologicl group $G$ is a 2-semitopological group, then $G$ is Hausdorff.

(2) Each compact 2-semitopological group is a topological group, see \cite[Theorem 6]{RE}.

(3) A $\sigma$-compact regular 2-semitopological group is ccc, see \cite[Corollary 4]{RE}.

(4) For any $T_{1}$-semitopological group $G$ and $n\in\mathbb{N}$, $G$ is an $n$-semitopological group if and only if $\bigcap_{U\in\mathcal{N}_{e}}U^{n}=\{e\}$, where $\mathcal{N}_{e}$ denotes the family of all neighborhoods of the neutral element $e$ of $G$.
\end{remark}

The following question is interesting. Indeed, we give an example to show that there exists a 2-semitopological group which is not a $3$-semitopological group in Section 2.

\begin{question}\label{q3}
For any $n\in \mathbb{N}\setminus\{1\}$, does there exists an $n$-semitopological group $G$ such that $G$ is not an ($n+1$)-semitopological group?
\end{question}

Moreover, it is natural to pose the following question by above remark.

\begin{question}\label{q2}
If $G$ is a locally compact $n$-semitopological group for some $n\in (\mathbb{N}\cup\{\infty\})\setminus\{1\}$, is $G$ a topological group?
\end{question}

In this paper, we give some partial answers to above question and discuss some generalized metric properties of $n$-semitopological groups, where $n\in(\mathbb{N}\cup\{\infty\})\setminus\{1\}$. The paper is organized as follows.

In Section 2, we mainly give some topological properties of $n$-semitopological groups ($n\in\mathbb{N})$. First, we give a Hausdorff quasi-topological group $G$ (thus a 2-semitopological group) such that $G$ is not a $3$-semitopological group. Moreover, we prove that (1) Each Hausdorff $2$-semitopological group is weakly $qg$-separated; (2) For any $m\in\mathbb{N}\setminus\{1\}$, a semitopological group $G$ is a $T_{1}$ $m$-semitopological group if and only if $S_{G}^{m}=\{(x_{1}, \ldots, x_{m})\in G^{m}: x_{1}\cdot\ldots \cdot x_{m}=e\}$ is closed in $G^{m}$.

In Section 3, we mainly discuss some generalized metric properties of $n$-semitopological groups ($n\in\mathbb{N})$. We prove that (1) each Hausdorff first-countable 2-semitopological group admits a coarser semi-metrizable topology; (2) each locally compact, Baire and $\sigma$-compact 2-semitopological group is a topological group; (3) the condensation of a 2-semitopological group topology is given.

In Section 4, we mainly consider some cardinal invariants of $n$-semitopological groups ($n\in\mathbb{N})$. We mainly prove that (1) if $G$ is a $T_{1}$ $m$-semitopological group and $H$ a compact closed neutral subgroup of $G$, where $m\in\mathbb{N}\setminus\{1\}$, then $G/H$ is an $(m$-$1)$-semitopological group; (2) if $G$ is a regular $\kappa$-Lindel\"{o}f $\kappa$-$\Sigma$ 2-semitopological group, then $G$ is a $\kappa$-cellular space.
Moreover, some interesting questions are posed.

The symbol $\mathbb{N}$ denotes the natural numbers. The letter $e$
denotes the neutral element of a group, and $I$ denotes the unit interval with usual topology. Put $\mathbb{N}^{\ast}=\mathbb{N}\cup\{\infty\}$. For a semitopological group $G$, we denote the family of all neighborhoods of the neutral element $e$ by $\mathcal{N}_{e}$. Readers may refer to \cite{AA, ER11, Gr} for notations and terminology not
explicitly given here.

\maketitle
\section{Some properties of $n$-semitopological groups}
In this section, we mainly discuss some properties of $n$-semitopological groups, and pose some questions about $n$-semitopological groups, where $n\in\mathbb{N}^{\ast}$. First, we give a partial answer to Question~\ref{q3}.

\begin{example}\label{e0}
There exists a Hausdorff quasi-topological group $G$ (thus a 2-semitopological group) such that $G$ is not a $3$-semitopological group.
\end{example}

\begin{proof}
We consider the strongest topology $\tau$ on the group of integers $G=\mathbb{Z}$ such that for every $z\in \mathbb{Z}$ the sequence $(z\pm n^2)_{n\in\omega}$ converges to $z$. We claim that $G$ is Hausdorff. Indeed, for each $m\in\mathbb{Z}$, put $F_{m}=\{m\pm n^2: n\in\omega\}$. Then each $F_{m}$ is compact in $\tau$. Let $\sigma$ be determined by the countable family of compact subsets $\{F_{m}: m\in\mathbb{N}\}$. Obviously, we have $\sigma\subset\tau$. Since for any distinct numbers $z,z'\in Z$, the sets $\{z\pm n^2:n\in \omega\}$ and $\{z'\pm n^2:n\in\omega\}$ has finite intersection, it is easy to see that $\sigma$ is a Hausdorff $k_\omega$-topology. Therefore, it is easy to check that $(G, \tau)$ is a Hausdorff quasi-topological group.
Next we claim that $1\in \bigcap\{U+U+U: 0\in U\in\tau\}.$

Indeed, it suffices to prove that for every even number $a>0$ there exist numbers $n, m, k>a$ such that $1=k^2+n^2-m^2$. Take $k=a^2+1$ and observe that $$k^2-1=(k-1)(k+1)=a^2(k+1)=m^2-n^2=(m-n)(m+n),$$ where $m,n$ can be found from the equation $m-n=a$ and $m+n=a(a^2+2)$. Then $m=\frac{a(a^2+3)}{2}>a$ and $n=\frac{a^3+a}{2}>a$. Thus, we have $1=k^2+n^2-m^2$.

Therefore, $(G, \tau)$ is a 2-semitopological group, but it is not a $3$-semitopological group.
\end{proof}

Next, we give some concepts in order to discuss some properties of $n$-semitopological groups.

Let $(G, \tau)$ be a semitopological group. The {\it paratopological group reflexion} $G^{pg}=(G, \tau^{pg})$ of $(G, \tau)$ we understand the group $G$ endowed with the strongest topology $\tau^{pg}\subset \tau$ turning $G$ into paratopological group. The {\it quasitopological group reflexion} $G^{qg}=(G, \tau^{qg})$ of $(G, \tau)$ we understand the group $G$ endowed with the strongest topology $\tau^{qg}\subset \tau$ turning $G$ into quasitopological group. Clearly, the following characteristic property holds:  the identity map $i: G\rightarrow G^{pg}$ is continuous and for every continuous group homomorphism $h: G\rightarrow H$ from $G$ into a paratopological group $H$ the homomorphism $h\circ i^{-1}: G^{pg}\rightarrow H$ is continuous. The situation of quasitopological group reflexion is similar. A subset $U$ of $G$ is called {\it $pg$-closed ($pg$-open)} if $U$ is closed ($pg$-open) in $G^{pg}$; a subset $U$ of $G$ is called {\it $qg$-closed ($qg$-open)} if $U$ is closed ($qg$-open) in $G^{qg}$. A semitopological group $G$ is called {\it $pg$-separated} ({\it $qg$-separated}) provided its group reflexion  $G^{pg}$ ($G^{qg}$) is Hausdorff.

First, we have the following two propositions.

\begin{proposition}\label{pp}
Let $G$ be a semitopological group, and let $\mathscr{B}$ be a neighborhood base of $e$. Then the family $\{U\cup U^{^{-1}}: U\in \mathscr{B}\}$ is a weak base base of $e$ in $G^{qg}$.
\end{proposition}

\begin{proof}
By \cite[Construction 1.3.8 and Theorem~1.3.10]{AA}, it is easy to verify that the family $\{U\cup U^{^{-1}}: U\in \mathscr{B}\}$ is a weak base base of $e$ in $G^{qg}$.
\end{proof}

\begin{proposition}\label{ppp}
Let $(G, \tau)$ be a semitopological group, and let $\mathscr{B}$ be a neighborhood base of $e$. Then the topology on $G$ generated by the family $\mathscr{F}_{e}=\{UU^{^{-1}}\cap U^{^{-1}}U: U\in \mathscr{B}\}$ is a quasi-topological group which is coarser than $G^{qg}$.
\end{proposition}

\begin{proof}
First, we prove that the topology on $G$ generated by the family $\mathscr{F}_{e}=\{UU^{-1}\cap U^{-1}U: U\in \mathscr{B}\}$ is a quasi-topological group. Indeed, by \cite[Construction 1.3.8, Proposition 1.3.9 and Theorem 1.3.10]{AA}, it suffices to prove that for any $UU^{-1}\cap U^{-1}U$ and any $x\in UU^{-1}\cap U^{-1}U$, there exists $W\in \mathscr{B}$ such that $x (WW^{^{-1}}\cap W^{-1}W)\subset UU^{-1}\cap U^{-1}U$, where $U\in\mathscr{B}$. Now pick any $x\in UU^{-1}\cap U^{-1}U$. Then there exist $y_{1}, z_{1}, y_{2}, z_{2}\in U$ such that $x=y_{1}z_{1}^{-1}=y_{2}^{-1}z_{2}$. Since $\mathscr{B}$ is a neighborhood base of $e$ in $G$, there exist $V_{1}, W_{1}\in \mathscr{B}$ such that $W_{1} \subset V_{1}\subset U$, $y_{1}V_{1}\subset U$, $V_{1}z_{1}\subset U$ and $z_{1}^{-1}W_{1}\subset V_{1}z_{1}^{-1}$. Hence $x=yz^{-1}\in yz^{-1}W_{1}W_{1}^{-1}\subset yV_{1}z^{-1}V_{1}^{-1}=yV (Vz)^{-1}\subset UU^{-1}$. Similarly, we can find $W_{2}\in \mathscr{B}$ such that $W_{2}\subset U$, $x=y_{2}^{-1}z_{2}\in y_{2}^{-1}z_{2}W_{2}^{-1}W_{2}\subset U^{-1}U$. Put $W=W_{1}\cap W_{2}$. Then we have $x(WW^{-1}\cap W^{-1}W)\subset UU^{-1}\cap U^{-1}U$.

Moreover, by the above proof, we have $x(W\cup W^{-1})\subset UU^{-1}\cap U^{-1}U$, which implies that $UU^{-1}\cap U^{-1}U$ is open in $G^{qg}$ by Proposition~\ref{pp}. Therefore, the topology on $G$ generated by the family $\mathscr{F}_{e}=\{UU^{^{-1}}\cap U^{^{-1}}U: U\in \mathscr{B}\}$ is coarser than $G^{qg}$.
\end{proof}

By Proposition~\ref{pp}, the following proposition is obvious.

\begin{proposition}
Let $G$ be a semitopological group. If $G^{pg}$ is $T_{1}$, then $G$ is an $\infty$-semitopological group; if $G^{qg}$ is Hausdorff, then $G$ is a $2$-semitopological group.
\end{proposition}

A space $X$ is said to be {\it weakly Hausdorff} if there exists a weak base $\mathscr{B}$ such that for any distinct points $x, y$ there exist $B_{1}, B_{2}\in\mathscr{B}$ such that $x\in B_{1}$, $y\in B_{2}$ and $B_{1}\cap B_{2}=\emptyset$. A semitopological group $G$ is called {\it weakly-$pg$-separated} ({\it weakly-$qg$-separated}) provided its group reflexion  $G^{pg}$ ($G^{qg}$) is weakly Hausdorff.

The following proposition shows that each Hausdorff $2$-semitopological group is weakly-$qg$-separated.

\begin{proposition}\label{p1}
Let $G$ be a Hausdorff $2$-semitopological group. Then $G$ is weakly-$qg$-separated.
\end{proposition}

\begin{proof}
Take any $g\neq e$. Since $G$ is a Hausdorff $2$-semitopological group, there exists an open neighborhood $U$ of $e$ such that $gU\cap U=\emptyset$ and $Ug\cap U=\emptyset$, $g\not\in U^{2}$. Moreover, it follows from \cite[Proposition 5 (4)]{RE} that there exists an open neighborhood $W\subset U$ of $e$ such that $\{g, g^{-1}\}\cap (W^{-1})^{2}=\emptyset$. Then $gW\cap W=\emptyset$, $Wg\cap W=\emptyset$, $g\not\in W^{2}$ and $g\not\in (W^{-1})^{2}$, hence $gW\cap (W\cup W^{-1})=\emptyset$ and $gW^{-1}\cap (W\cup W^{-1})=\emptyset$. Therefore, we have $g(W\cup W^{-1})\cap (W\cup W^{-1})=\emptyset$. Thus $G$ is weakly-$qg$-separated by Proposition~\ref{pp}.
\end{proof}

Let $X$ be a space, and let $(Homeop(X), \tau_{p})$ be the group of all
homeomorphisms of $X$ onto itself, with the pointwise convergence topology. Then $(Homeop(X), \tau_{p})$ is a semitopological group, but it need not be a topological group, see \cite[Example 1.2.12]{AA}. It is well-known that if $X$ is a discrete space or $X=I$ then $(Homeop(X), \tau_{p})$ is a topological group, see \cite[Exercises 1.2.k]{AA}. Therefore, the following question is interesting.

\begin{question}
How to given a characterization $\mathcal{P}$ of the space $X$ such that $(Homeop(X), \tau_{p})$ is a $2$-semitopological group if and only if $X$ has the property $\mathcal{P}$?
\end{question}

\begin{proposition}
Let $X$ be a $T_{2}$ locally compact space and $(Homeop(X), \tau_{c})$ the group of all
homeomorphisms of $X$ onto itself, with the compact-open topology. Then $(Homeop(X), \tau_{c})$ is an $\infty$-semitopological group.
\end{proposition}

\begin{proof}
Since $X$ is a $T_{2}$ locally compact space, it is well known that $(Homeop(X), \tau_{c})$ is a paratopological group, hence $(Homeop(X), \tau_{c})$ is an $\infty$-semitopological group.
\end{proof}

In particular, if $X$ is a $T_{2}$ compact space, then $(Homeop(X), \tau_{c})$ is a topological group, hence it is an $\infty$-semitopological group. However, the following question is still open.

\begin{question}
How to given a characterization $\mathcal{P}$ of the space $X$ such that $(Homeop(X), \tau_{c})$ is a $2$-semitopological group if and only if $X$ has the property $\mathcal{P}$?
\end{question}

Let $(X, \tau)$ be a space. A subset $A$ of $X$ is called {\it regular open} if $A=\mbox{int}(\overline{A})$. The family of all regular open sets forms a base for a smaller topology $\tau_{s}$ on $X$, which is called the {\it semi-regularization} of $\tau$. The following question is still unknown for us.

\begin{question}
Let $G$ be an $m$-semitopological group for some $m\in\mathbb{N}$. Is the semiregularization $G_{sr}$  an $m$-semitopological group? What if we assume the space to be $\infty$-semitopological group?
\end{question}

Next, we discuss some important properties of $m$-semitopological groups for some $m\in\mathbb{N}$.

\begin{theorem}
Let $G$ be a semitopological group and $m\in\mathbb{N}^{\ast}$. If one of the following conditions is satisfied, then $G$ is an $m$-semitopological group.
\begin{enumerate}
\smallskip
\item $G$ is a paratopological group;

\smallskip
\item $G$ is a subgroup of an $m$-semitopological group;

\smallskip
\item $G$ is the product of $m$-semitopological groups;

\smallskip
\item there exists a continuous isomorphism of $G$ onto a $T_{1}$ $m$-semitopological group.
\end{enumerate}
\end{theorem}

\begin{proof}
Obviously, (1) and (2) hold.

(3) First, we consider $m\in\mathbb{N}^{\ast}\setminus\{\infty\}$. Let $\{G_{\alpha}: \alpha\in A\}$ be a family of $m$-semitopological groups such that $G=\Pi_{\alpha\in A}G_{\alpha}$. Take any $g=(g_{\alpha})_{\alpha\in A}$ with $e\not\in \overline{\{g\}}$. It is obvious that there exists $\beta\in A$ such that $e_{\beta}\not\in \overline{\{g_{\beta}\}}$, then there exists an open neighborhood $U_{\beta}$ of $e_{\beta}$ in $G_{\beta}$ such that $g_{\alpha}\not\in U_{\beta}^{m}$. Put $U=U_{\beta}\times \Pi_{\alpha\in A\setminus\{\beta\}}G_{\alpha}$. Then $g\not\in U^{m}$. The proof of the case of $m=\infty$ is similar.

(4) First, we consider $m\in\mathbb{N}^{\ast}\setminus\{\infty\}$. Suppose that $\phi: G\rightarrow H$ is a continuous isomorphism of the group $G$ onto a $T_{1}$ $m$-semitopological group $H$. Take any $g\neq e$ in $G$. Then there exists an open neighborhood $W$ of the neutral element in $H$ such that $\phi(g)\not\in W^{m}$. Put $V=\phi^{-1}(W)$. Hence $g\not\in V^{m}$. The proof of the case of $m=\infty$ is similar.
\end{proof}

Let $G$ be a group and any integer number $m\geq 2$. We denote $$S_{G}^{m}=\{(x_{1}, \ldots, x_{m})\in G^{m}: x_{1}\cdot\ldots \cdot x_{m}=e\}, E_{G}^{m}=\bigcap_{U\in\mathcal{N}_{e}}\overline{(U^{-1})^{m-1}}.$$ The following theorem gives some characterizations of $m$-semitopological groups for each $m\in\mathbb{N}^{\ast}\setminus\{1, \infty\}$.

\begin{theorem}\label{t1}
Let $G$ be a semitopological group and $m\in\mathbb{N}^{\ast}\setminus\{1, \infty\}$. Then we have
\begin{enumerate}
\smallskip
\item $$\overline{\{e\}}\subset E_{G}^{m}=\bigcap_{U\in\mathcal{N}_{e}}(U^{-1})^{m};$$

\smallskip
\item $G$ is an $m$-semitopological group if and only if $E_{G}^{m}=\overline{\{e\}}$;

\smallskip
\item $\overline{S_{G}^{m}}=\mathfrak{m}^{-1}(E_{G}^{m})$, where $\mathfrak{m}$ is the multiplication in the group $G$;

\smallskip
\item the following statements are equivalent:

\smallskip
(i) $G$ is a $T_{1}$ $m$-semitopological group;

\smallskip
(ii) $E_{G}^{m}=\{e\}$;

\smallskip
(iii) $S_{G}^{m}$ is closed in $G^{m}$.
\end{enumerate}
\end{theorem}

\begin{proof}
(1) From \cite[Proposition 4]{RE}, it follows that $\overline{\{e\}}\subset E_{G}^{m}\subset\bigcap_{U\in\mathcal{N}_{e}}(U^{-1})^{m}$. Take any $g\in G\setminus E_{G}^{m}$. Then there exists $U\in \mathcal{N}_{e}$ such that $gU\cap (U^{-1})^{m-1}=\emptyset$, hence $g\not\in (U^{-1})^{m}$. Thus $E_{G}^{m}=\bigcap_{U\in\mathcal{N}_{e}}(U^{-1})^{m}.$

(2) Let $G$ be an $m$-semitopological group. Then $\overline{\{e\}}\subseteq E_{G}^{m}$ by (1). Take any $g\not\in \overline{\{e\}}$. Hence $e\not\in\overline{\{g^{-1}\}}$. Since $G$ is an $m$-semitopological group, there exists $U\in\mathcal{N}_{e}$ such that $g^{-1}\not\in U^{m}$, hence $g\not\in (U^{-1})^{m}\supseteq E_{G}^{m}$. Therefore, $E_{G}^{m}\subseteq\overline{\{e\}}$.

Now suppose $E_{G}^{m}=\overline{\{e\}}$. Take any $g\neq e$ with $e\not\in\overline{\{g\}}$. Then $g^{-1}\not\in\overline{\{e\}}=E_{G}^{m}$. From (1), it follows that there exists $U\in \mathcal{N}_{e}$ such that $g^{-1}\not\in (U^{-1})^{m}$. Then $g\not\in U^{m}$.

(3) Let $(x_{1}, \ldots, x_{m})\in G^{m}$. Clearly, we have $$(x_{1}, \ldots, x_{m})\in\overline{S_{G}^{m}}\Leftrightarrow (Ux_{1}\times \ldots \times Ux_{m})\cap S_{G}^{m}\neq\emptyset$$ for any $U\in\mathcal{N}_{e},$ that is $e\in Ux_{1}\ldots Ux_{m}$ for any $U\in\mathcal{N}_{e}$. Hence $$(x_{1}, \ldots, x_{m})\in\overline{S_{G}^{m}}\Leftrightarrow e\in x_{1}\ldots x_{m}U^{m}$$ for any $U\in\mathcal{N}_{e}$, then $$(x_{1}, \ldots, x_{m})\in\overline{S_{G}^{m}}\Leftrightarrow x_{1}\ldots x_{m}\in (U^{-1})^{m}$$ for any $U\in\mathcal{N}_{e}$. By (1), we have $(x_{1}, \ldots, x_{m})\in\overline{S_{G}^{m}}\Leftrightarrow x_{1}\ldots x_{m}\in E_{G}^{m}$.

(4) From (2), it follows that (i) $\Leftrightarrow$ (ii). (ii) $\Leftrightarrow$ (iii) since $S_{G}^{m}=\mathfrak{m}^{-1}(e)$.
\end{proof}

By Theorem~\ref{t1} and the definition of $\infty$-semitopological group, we have the following theorem.

\begin{theorem}
Let $G$ be a semitopological group. Then we have
\begin{enumerate}
\smallskip
\item $G$ is an $\infty$-semitopological group if and only if $E_{G}^{m}=\overline{\{e\}}$ for each $m\in\mathbb{N}$;

\smallskip
\item $\overline{S_{G}^{m}}=\mathfrak{m}^{-1}(E_{G}^{m})$ for each $m\in\mathbb{N}$;

\smallskip
\item the following statements are equivalent:

\smallskip
(i) $G$ is a $T_{1}$ $\infty$-semitopological group;

\smallskip
(ii) $E_{G}^{m}=\{e\}$ for each $m\in\mathbb{N}$;

\smallskip
(iii) $S_{G}^{m}$ is closed in $G^{m}$ for each $m\in\mathbb{N}$.
\end{enumerate}
\end{theorem}

Suppose that $X$ and $Y$ are spaces. We say that the mapping $f: X\rightarrow Y$ is {\it topology-preserving} if the following conditions are satisfied:

\smallskip
(1) $f$ is surjective, continuous, open and closed;

\smallskip
(2) a subset $U$ of $X$ is open if and only if $U=f^{-1}(f(U))$ and $f(U)$ is open.

The following proposition shows that the topology-preserving mappings can preserve and inversely preserve for the class of $m$-semitopological groups, where $m\in\mathbb{N}^{\ast}$.

\begin{proposition}
Let $G$ and $H$ be two semitopological groups, and let $\phi: G\rightarrow H$ be a topology-preserving homomorphism. Then $G$ is an $m$-semitopological group if and only if $H$ is an $m$-semitopological group, where $m\in\mathbb{N}^{\ast}$.
\end{proposition}

\begin{proof}
We divide the proof into the following two cases.

\smallskip
{\bf Case 1} $m\in\mathbb{N}^{\ast}\setminus\{\infty\}$.

\smallskip
Assume that $G$ is an $m$-semitopological group. Take any $h\neq e_{H}$ and $e_{H}\not\in\overline{\{h\}}$ in $H$. Then there exists $g\in G$ such that $\phi(g)=h$. Clearly, $e_{G}\not\in\overline{\{g\}}$ in $G$ since $e_{H}\not\in\overline{\{h\}}$ and $\phi$ is a topology-preserving mapping. Since $G$ is an $m$-semitopological group, there exists an open neighborhood $U$ of $e_{G}$ such that $g\not\in U^{m}$. We claim that $h\not\in (\phi(U))^{m}$. Indeed, suppose $h\in (\phi(U))^{m}$. Then $\phi^{-1}(h)\cap \phi^{-1}((\phi(U))^{m})\neq\emptyset$. Since $\phi^{-1}((\phi(U))^{m})=U^{m}$, it follows that $U^{m}\cap \phi^{-1}(h)\neq\emptyset$, then $\phi^{-1}(h)\subset U^{m}$ since $\phi^{-1}(h)$ is antidiscrete. Hence $g\in U^{m}$, which is a contradiction. Therefore, $h\not\in (\phi(U))^{m}$. Thus $H$ is an $m$-semitopological group.

Assume that $H$ is an $m$-semitopological group. Take any $g\neq e_{G}$ and $e_{G}\not\in\overline{\{g\}}$ in $G$. Since $\phi$ is a topology-preserving mapping, it follows from \cite[Proposition 1]{QL} that $e_{H}\not\in\overline{\{\phi(g)\}}$, hence there exists an open neighborhood $V$ of $e_{H}$ in $H$ such that $\phi(g)\not\in V^{m}$. Then $\phi^{-1}(\phi(g))\cap \phi^{-1}(V^{m})=\emptyset$, hence $g\not\in (\phi^{-1}(V))^{m}$. Therefore, $G$ is an $m$-semitopological group.

\smallskip
{\bf Case 2} $m=\infty$.

\smallskip
The proof is similar to Case 1.
\end{proof}

Finally, we consider the topological direct limit of $m$-semitopological groups, $m\in \mathbb{N}^{\ast}$. First, we recall the following concept.

\begin{definition}
Given a tower $$X_{0}\subset X_{1}\subset X_{2}\subset \ldots\subset X_{n}\subset\ldots$$ of spaces, the union $X=\bigcup_{n\in\mathbb{N}}X_{n}$ endowed with the strongest topology making each inclusion map $X_{n}\rightarrow X$ continuous is called the {\it topological direct limit} of the tower $(X_{n})_{n\in\mathbb{N}}$ and is denoted by $\underrightarrow{\lim}X_{n}$.
\end{definition}

Let $\{G_{n}: n\in\mathbb{N}\}$ be a tower of semitopological groups. From \cite[Proposition 1.1]{TSH} that $G=\underrightarrow{\lim}G_{n}$ is a semitopological group. Moreover, if each $G_{n}$ is a quasitopological group, then $G$ is a quasitopological group by \cite[Proposition 1.1]{TSH} again. However, there exists a tower $\{G_{n}: n\in\mathbb{N}\}$ of topological groups such that $G$ is not a paratopological group, see \cite[Example 1.2]{TSH}. Therefore, we have the following question.

\begin{question}
Let $\{G_{n}: n\in\mathbb{N}\}$ be a tower of $m$-semitopological groups (resp., $\infty$-semitopological groups), where $m\geq 2$. Is $G=\underrightarrow{\lim}G_{n}$ an $m$-semitopological group (resp., $\infty$-semitopological group)?
\end{question}

The following two results are obvious.

\begin{theorem}
Let $\{H_{n}: n\in\mathbb{N}\}$ be a sequence of $m$-semitopological groups (resp., $\infty$-semitopological groups), where $m\geq 2$. Then both the $\sigma$-product and $\Sigma$-product of $\prod_{i\in\mathbb{N}}H_{n}$ are $m$-semitopological groups (resp., $\infty$-semitopological groups).
\end{theorem}

\begin{corollary}
Let $\{H_{n}: n\in\mathbb{N}\}$ be a sequence of $m$-semitopological groups (resp., $\infty$-semitopological groups), where $m\geq 2$. Then $G=\underrightarrow{\lim}G_{n}$ is an $m$-semitopological groups (resp., $\infty$-semitopological groups), where $G_{n}=\prod_{i\leq n}H_{i}$ and each $G_{n}$ is identified as a subspace of $G_{n+1}$ for each $n\in\mathbb{N}$.
\end{corollary}

For closing this section, we give the following proposition.

\begin{proposition}\label{p0}
Let $\{G_{n}: n\in\mathbb{N}\}$ be a tower of semitopological groups. If each $G_{n}$ is $T_{1}$, then $G=\underrightarrow{\lim}G_{n}$ is $T_{1}$.
\end{proposition}

\begin{proof}
It suffices to prove that $\{e\}$ is closed in $G$. Since each $G_{n}$ is $T_{1}$, it follows that $\{e\}$ is closed in each $G_{n}$. Therefore, $\{e\}$ is closed in $G$.
\end{proof}

\maketitle
\section{generalized metric properties of $n$-semitopological groups}
In this section, we mainly discuss some generalized metric properties of $n$-semitopological groups, such as, weakly first-countable, semi-metrizable,  symmetrizable and etc. First, we recall a concept.

\begin{definition}
Let $\mathscr{P}=\bigcup_{x\in X}\mathscr{P}_{x}$ be a cover of a
space $X$ such that for each $x\in X$, (a) if $U,V\in
\mathscr{P}_{x}$, then $W\subset U\cap V$ for some $W\in
\mathscr{P}_{x}$; (b) the family $\mathscr{P}_{x}$ is a network of $x$ in $X$,
i.e., $x\in\bigcap\mathscr{P}_x$, and if $x\in U$ with $U$ open in
$X$, then $P\subset U$ for some $P\in\mathscr P_x$.
The family $\mathscr{P}$ is called a {\it weak base} for $X$ \cite{Ar} if, for every $A\subset X$, the set $A$ is open in $X$ whenever for each $x\in A$ there exists $P\in
\mathscr{P}_{x}$ such that $P\subset A$.
The space $X$ is {\it weakly first-countable} if $\mathscr{P}_{x}$ is countable for each
$x\in X$.
\end{definition}

From \cite{NP}, it follows that all weakly first-countable paratopological groups are first-countable; moreover, there exists a Hausdorff weakly first-countable quasitopological group is not first-countable \cite[Example 2.1]{LC}. Therefore, we have the following question.

\begin{question}
Let $G$ be an $n$-semitopological group (resp., $\infty$-semitopological group), where $n\geq 2$. If $G$ is weakly first-countable, when is $G$ a first-countable space?
\end{question}

Let us recall that a function $d: X\times X\rightarrow [0, +\infty)$ on a set $X$ is a {\it symmetric} if for every points $x, y$ the following two conditions are satisfied: (1) $d(x, y)=0$ if and only if $x=y$; (2) $d(x, y)=(d(y, x)$. Fro each $x\in X$ and $\varepsilon>0$, denote by $B(x, \varepsilon)=\{y\in X: d(x, y)<\varepsilon\}$. Then

$\bullet$ a space $X$ is {\it symetrizable} if there is a symmetric $d$ on $X$ such that $U\subset X$ is open if and only if for each $x\in U$, there exists $\varepsilon>0$ with $B(x, \varepsilon)\subset U$;

$\bullet$ a space $X$ is {\it semi-metrizable} if there is a symmetric $d$ on $X$ such that for each $x\in X$, the family $\{B(x, \varepsilon): \varepsilon>0\}$ forms a neighborhood base at $x$;

$\bullet$ a space $X$ is called a {\it sub-symmetrizable space} if it admits a coarser symmetrizable topology;

$\bullet$ a space $X$ is called a {\it subsemi-metrizable space} if it admits a coarser semi-metrizable topology.

Every symmetrizable space is weakly first-countable, and a space is semi-metrizable if and only if it is first-countable and symmetrizable, see \cite{Ar}.

\begin{theorem}\label{t4}
Let $(G, \sigma)$ be a $T_{1}$ weakly first-countable semitopological group. Then $(G, \sigma)$ is sub-symmetrizable.
\end{theorem}

\begin{proof}
Since $G$ is weakly first-countable, we may assume that $\mathcal{P}_{e}=\{P_{n}(e): n\in\mathbb{N}\}$ be a countable weak base at $e$ for $G$, where $P_{1}(x)=G$ and $P_{n+1}(x)\subset P_{n}(x)$ for each $n\in\mathbb{N}$. For each $x\in G$, let $\mathcal{P}_{x}=\{xP_{n}(e): n\in\mathbb{N}\}$. Put $\mathcal{P}=\bigcup_{x\in G}\mathcal{P}_{x}$. Then $\mathcal{P}$ is a countable weak base for $G$. For each $n\in \mathbb{N}$, put $W_{n}(e)=P_{n}(e)\cup (P_{n}(e))^{-1}$; then define a function $d: G\times G\rightarrow \mathbb{R}$ by setting $d(x, y)=\inf\{\frac{1}{n}: x^{-1}y\in W_{n}(e)\}$. We claim that $d$ is a symmetric on $G$. Indeed, it is obvious that $d(x, y)=d(y, x)$ for any $x, y\in G$. Now suppose that $d(x, y)=0$ for $x, y\in G$. Then from our assumption, it follows that $x^{-1}y\in W_{n}(e)$ for any $n\in\mathbb{N}$, hence $x^{-1}y\in P_{n}(e)\cup (P_{n}(e))^{-1}$ for any $n\in\mathbb{N}$. Assume that $x\neq y$. Then since $(G, \sigma)$ is $T_{1}$, it follows that $e\not\in\overline{\{x^{-1}y\}}$ and $e\not\in\overline{\{y^{-1}x\}}$. Then there exists $k\in \mathbb{N}$ such that $x^{-1}y\not\in P_{k}(e)$ and $y^{-1}x\not\in P_{k}(e)$, hence $x^{-1}y \not\in P_{k}(e)\cup(P_{k}(e))^{-1}$. This is a contradiction. Therefore, we have $x=y$.

Clearly, for any $n\in\mathbb{N}$ and $x\in G$, we have $xW_{n+1}(e)=B(x, \frac{1}{n})$. The topology $\tau$ which is inducted by the symmetric $d$ on $G$ is coarser than $\sigma$. Therefore, $(G, \sigma)$ is sub-symmetrizable.
\end{proof}

It is well known that each first-countable paratopological group is submetrizable. However, the Sorgenfrey line is a first-countable $\infty$-semitopological group which is not symmetrizable. Therefore, the following question is natural.

\begin{question}
Let $(G, \sigma)$ be a $T_{1}$ weakly first-countable 2-semitopological group. When is $(G, \sigma)$ symmetrizable?
\end{question}

If we improve the conditions in Theorem~\ref{t4}, then we have the following result.

\begin{theorem}
Let $(G, \sigma)$ be a Hausdorff first-countable 2-semitopological group. Then $(G, \sigma)$ admits a semi-metrizable quasitopological group topology which is coarser than the weakly-$qg$-separated quasitopological group reflexion $G^{qg}$ of $(G, \sigma)$.
\end{theorem}

\begin{proof}
Let $\{U_{n}: n\in\mathbb{N}\}$ be a countable neighborhood base of $e$ such that $U_{n+1}\subset U_{n}$ for each $n\in\mathbb{N}$. For any $g\in G$, put $\mathscr{B}=\{g(U_{n}U_{n}^{-1}\cap U_{n}^{-1}U_{n}): n\in\mathbb{N}, g\in G\}$. Let $\tau$ be the topology generated by the neighborhood system $\mathscr{B}$. By Proposition~\ref{ppp}, $(G, \tau)$ is a first-countable quasitopological group and $\tau$ is coarser than the topology of $\sigma$. By Propositions~\ref{ppp} and~\ref{p1}, $(G, \tau)$ is coarser than the weakly-$qg$-separated quasitopological group reflexion $G^{qg}$ of $(G, \sigma)$.

Since $(G, \sigma)$ is Hausdorff, it follows that $(G, \tau)$ is $T_{1}$. By the proof of \cite[Theorem 2.1]{Li} and \cite[Corollary 1.4]{CG1970}, $(G, \tau)$ is semi-metrizable.
\end{proof}

Next we recall some concepts, and then pose Question~\ref{q1}.

\begin{definition}
Let $X$ be a space and $\{\mathscr{P}_{n}\}_{n}$ a sequence of
collections of open subsets of $X$.
\begin{enumerate}
\item $X$ is called {\it developable} for $X$ if $\{\mbox{st}(x, \mathscr{P}_{n})\}_{n}$
is a neighborhood base at $x$ in $X$ for each
point $x\in X$.

\item $X$ is called {\it Moore}, if $X$ is regular and developable.

\item $X$ is called a $\mbox{wM}$-{\it space} if for each $x\in X$ and a sequence $\{x_{n}\}_{n}$ whenever $x_{n}\in\mbox{st}^{2}(x, \mathscr{U}_{n})$ then the set $\{x_{n}: n\in\mathbb{N}\}$ has a cluster point in $X$.
\end{enumerate}
\end{definition}

In \cite{LC}, C. Liu proved that each regular paratopological group $G$, in which each singleton is a $G_{\delta}$-set, is metrizable if $G$ is a $\mbox{wM}$-space, and posed that if we can replace ``paratopological group'' with ``semitopological group''. Then R. Shen in \cite{Sr} gave a Moore quasitopological group which is not metrizable. Therefore, a Moore $\infty$-semitopological groups may not be metrizable. Hence we have the following question.

\begin{question}\label{q1}
Let $G$ be an $n$-semitopological group (resp., $\infty$-semitopological group), where $n\geq 2$. If $G$ is a $\mbox{wM}$-space in which each singleton is a $G_{\delta}$-set, is $G$ metrizable?
\end{question}

Next we give a partial answer to Question~\ref{q2}. First, we recall some concepts.

Let $X$ be a space. Then

\smallskip
(1) $X$ is said to be {\it locally compact} if for any point $x\in X$ there exists a compact neighborhood $C$ of $x$;

\smallskip
(2) $X$ is said to be {\it $\sigma$-compact} if $X=\bigcup_{n\in\mathbb{N}}K_{n}$, where each $K_{n}$ is compact;

\smallskip
(3) $X$ is said to be {\it Baire} if $X=\bigcup_{n\in\mathbb{N}}A_{n}$ then there exists $n\in\mathbb{N}$ such that the interior of $\overline{A_{n}}$ is nonempty.

\begin{theorem}\label{t0}
Each locally compact, Baire and $\sigma$-compact 2-semitopological group is a topological group.
\end{theorem}

\begin{proof}
Let $(G, \tau)$ be a locally compact, Baire and $\sigma$-compact 2-semitopological group, and let $H=\overline{\{e\}}$. Clearly, $H$ is a normal closed antidiscrete subgroup. Since the quotient mapping $\phi: G\rightarrow G/H$ is a topology-preserving homomorphism, it follows that the the quotient group $\widehat{G}=G/H$ is a $T_{1}$ locally compact, Baire and $\sigma$-compact 2-semitopological group. By \cite[Proposition 7]{RE}, $\widehat{G}$ is a topological group if and only if $G$ is a topological group. Therefore, it suffices to prove that $\widehat{G}$ is a topological group. Moreover, since $\widehat{G}$ is a $T_{1}$ 2-semitopological group, it follows from
\cite[Proposition 6.4 (a) and (b)]{RE} that $Sym(\widehat{G})$ is closed in $(\widehat{G})^{2}$ and $Sym(\widehat{G})$ is a Hausdorff locally compact $\sigma$-compact quasitopological group. From Ellis theorem \cite[Theorem 2]{ER} that $Sym(\widehat{G})$ is a topological group. Let $\widehat{\tau}$ and $\widehat{\tau}_{Sym}$ be the topologies of $\widehat{G}$ and $Sym(\widehat{G})$ respectively. By  Ellis theorem \cite[Theorem 2]{ER} again, it suffices to prove that $\widehat{G}$ is Hausdorff.

Take any $e\neq g\in G$. Since $\widehat{G}$ is a $T_{1}$ 2-semitopological group, it follows from \cite[Proposition 5(4)]{RE} that there exists $U\in \mathcal{N}(e)$ such that $g\not\in \overline{U^{-1}}$, where $\mathcal{N}(e)$ is the neighborhood of $e$ in $(G, \widehat{\tau})$. We claim that $e\in\mbox{Int}\overline{U^{-1}}$ in $(G, \widehat{\tau})$. Indeed, since $U^{-1}\in \widehat{\tau}_{Sym}$, there exists a symmetric open neighborhood $V$ of $e$ in $Sym(\widehat{G})$ such that $V^{2}\subset U^{-1}$. Since $Sym(\widehat{G})$ is $\sigma$-compact, there exists a countable subset $A$ of $G$ such that $G=\bigcup\{aV: a\in A\}$, then there exists $a\in A$ such that $\mbox{Int}\overline{aV}\neq\emptyset$ in $(G, \widehat{\tau})$ because $\widehat{G}$ is a Baire space. Then $\mbox{Int}\overline{V}\neq\emptyset$ in $(G, \widehat{\tau})$. Take any $v\in V\cap \mbox{Int}\overline{V}$. Hence $e\in\mbox{Int}\overline{v^{-1}V}\subset \overline{V^{2}}\subset \overline{U^{-1}}$, which shows that $e\in\mbox{Int}\overline{U^{-1}}$. Put $W=\widehat{G}\setminus \overline{U^{-1}}$ and $O=\mbox{Int}\overline{U^{-1}}$. Clearly, $W\cap O=\emptyset$, $g\in W$ and $e\in O$. Moreover, $W$ and $O$ are open in $\widehat{G}$. Therefore, $\widehat{G}$ is Hausdorff.
\end{proof}

\begin{remark}
(1) There exists a locally compact, Baire and $\sigma$-compact semitopological group $G$ such that $G$ is not an 2-semitopological group. Indeed, let $\tau$ be the cofinite topology on a uncountable group $H$. Suppose $G$ is the Tychonoff product of $H$ and the Euclidean space $\mathbb{R}$, then $G$ is a locally compact and $\sigma$-compact semitopological group. Clearly, $H$ is a Baire space, hence $G$ is Baire by \cite[3.9.J(c)]{ER11}. However, $G$ is not a 2-semitopological group since $H$ is not a 2-semitopological group.

(2) There exists a Hausdorff sequentially compact $\infty$-semitopological group $G$ which is not a paratopological group, see \cite[Example 3]{RA1}.
\end{remark}

Clearly, a compact semitopological group may not be a Baire space, such as any cofinite topology on a countable infinite group. Therefore, we have the following question.

\begin{question}
Is each compact 2-semitopological group a Baire space?
\end{question}

From \cite[Theorem 6]{RE}, each compact 2-semitopological group is a topological group, hence each compact $T_{0}$ 2-semitopological group is a Baire space.

Finally, we consider the condensation of 2-semitopological group topologies. First, we give some propositions and lemmas.

\begin{definition}
A family $\mathcal{P}$ of subsets of
a space $X$ is called a {\it network} for $X$ if for each $x\in X$ and neighborhood $U$ of $x$ there exists $P\in \mathcal{P}$ such that $x\in P\subset U$. The
infimum of the cardinalities of all networks of $X$ is denoted by $nw(X)$.
\end{definition}

The following proposition is obvious.

\begin{proposition}\label{p2}
Let $G$ be a semitopological group and $nw(G)\leq\kappa$, where $\kappa$ is some infinite cardinal. Then $nw(G^{qg})\leq\kappa$.
\end{proposition}

\begin{proposition}\label{p3}
Let $\tau$ and $\sigma$ be two topologies on group $G$ such that $(G, \tau)$ and $(G, \sigma)$ are semitopological groups with $w((G, \tau))\leq\kappa$ and $w((G, \sigma))\leq\kappa$, where $\kappa$ is some infinite cardinal. Then $w(G, \tau\vee\sigma)\leq\kappa$.
\end{proposition}

\begin{proof}
Let $\mathcal{B}_{1}$ and $\mathcal{B}_{2}$ be bases for $(G, \tau)$ and $(G, \sigma)$ respectively such that $|\mathcal{B}_{1}|\leq\kappa$ and $|\mathcal{B}_{2}|\leq\kappa$. Put $\mathcal{B}=\{U\cap V: U\in \mathcal{B}_{1}, V\in \mathcal{B}_{2}\}$. It is easily verified that $\mathcal{B}$ is a base for $\tau\vee\sigma$ and $|\mathcal{B}|\leq\kappa$. Therefore, $w(G, \tau\vee\sigma)\leq\kappa$.
\end{proof}

\begin{lemma}\label{l1}
Suppose that $\kappa$ is an infinite cardinal, $X$ is a group, $\tau$ is a Hausdorff (resp., regular, Tychonof) $2$-semitopological group topology on $X$ that has a network weight $\leq\kappa$ and $\tau^{\prime}$ is a topology on $X$ that has weight $\leq\kappa$ such that $\tau^{\prime}\subset\tau$. Then one can find a topology $\tau^{\ast}$ on $X$ with the following properties:

\smallskip
(i) $\tau^{\prime}\subset\tau^{\ast}\subset\tau$;

\smallskip
(ii) $w(X, \tau^{\ast})\leq\kappa$;

\smallskip
(iii) $(X, \tau^{\ast})$ is a Hausdorff (resp., regular, Tychonof) $2$-semitopological group.
\end{lemma}

\begin{proof}
We first prove the case of Hausdorff. By \cite[Lemma 4]{HC}, there exists a Hausdorff semitopological group topology $\sigma$ on $X$ such that $\tau^{\prime}\subset\sigma\subset\tau$ and $w(X, \sigma)\leq\kappa$. Then it follows from Proposition~\ref{p2} that $X^{qg}$ has a network weight $\leq\kappa$. Then one can find a $T_{1}$ quasitopological group topology $\delta$ on $X$ such that $\delta\subset \tau^{qg}$ and $w(X, \delta)\leq\kappa$ by \cite[Theorem 1]{SH}. Clearly, $(X, \delta)$ is a $2$-semitopological group by \cite[Theorem 5]{RE}. Now put $\tau^{\ast}=\sigma\vee \delta$. Then $\tau^{\ast}\subset\tau$ and $\tau^{\ast}$ is a Haudorff 2-semitopological group topology on $X$. By Proposition~\ref{p3}, $w(X, \tau^{\ast})\leq\kappa$. Moreover, we have $\tau^{\prime}\subset\tau^{\ast}\subset\tau$.

If $\tau$ is regular (Tychonof), then it follows from the above proof and \cite[Lemma 3]{HC} that there exists a topology $\tau^{\ast}$ on $X$ which has the properties of (i) and (ii) and $(X, \tau^{\ast})$ is a regular (Tychonof) $2$-semitopological group.
\end{proof}

Now we can prove the main theorem.

\begin{theorem}\label{t5}
Suppose that $\kappa$ is an infinite cardinal, $X$ is a group, $\tau$ is a Hausdorff (resp., regular, Tychonof) $2$-semitopological group topology on $X$ that has a network weight $\leq\kappa$. Then there exists a condensation $i: (X, \tau)\rightarrow (X, \tau^{\ast})$, where $\tau^{\ast}$ is a Hausdorff (resp., regular, Tychonof) $2$-semitopological group topology $\tau^{\ast}$ on $X$ such that $\tau^{\ast}\subset\tau$ and $w(X, \tau^{\ast})\leq\kappa$.
\end{theorem}

\begin{proof}
Since $X$ is Hausdorff (resp., regular, Tychonof) and has a network weight $\leq\kappa$, it follows from \cite[Lemma 3.1.8]{ER11} that there exists a Hausdorff space $(X, \tau_{0})$ such that $w(X, \tau_{0})\leq\kappa$. Now, it follows from Lemma~\ref{l1} that there exists a Hausdorff (resp., regular, Tychonof) $2$-semitopological group topology $\tau^{\ast}$ on $X$ such that $\tau^{\ast}\subset\tau$ and $w(X, \tau^{\ast})\leq\kappa$.
\end{proof}

By Theorem~\ref{t5}, we have the following corollary.

\begin{corollary}
Suppose that $\kappa$ is an infinite cardinal, $X$ is a group, $\tau$ is a Hausdorff (resp., regular, Tychonof) $2$-semitopological group topology on $X$ with a countable network. Then there exists a condensation $i: (X, \tau)\rightarrow (X, \tau^{\ast})$, where $\tau^{\ast}$ is a Hausdorff (resp., regular, Tychonof) second-countable $2$-semitopological group topology $\tau^{\ast}$ on $X$ such that $\tau^{\ast}\subset\tau$.
\end{corollary}

However, the following question is still unknown for us.

\begin{question}
Suppose that $\kappa$ is an infinite cardinal, $X$ is a group, $\tau$ is a Hausdorff (resp., regular, Tychonof) $m$-semitopological group topology on $X$ that has a network weight $\leq\kappa$, where $m\in\mathbb{N}^{\ast}\setminus\{2\}$. Can we find a Hausdorff (resp., regular, Tychonof) $m$-semitopological group topology $\tau^{\ast}$ on $X$ such that $\tau^{\ast}\subset\tau$ and $w(X, \tau^{\ast})\leq\kappa$?
\end{question}

\maketitle
\section{Cardinal invariants of $n$-semitopological groups}
In this section, we mainly consider some cardinal invariants of $n$-semitopological groups. Moreover, some interesting questions are posed. First, we recall some concepts.

Let $\kappa$ be an ordinal. A semitopological group $G$ is left (right) $\kappa$-narrow if for each open set $U$ there exists a set $A\subset G$ such that $|A|\leq\kappa$ and $AU=G$ ($UA=G$). Put $$\mbox{In}_{l}(G)=\min\{\kappa: G\ \mbox{is left}\ \kappa\ \mbox{-narrow}\}, \mbox{In}_{r}(G)=\min\{\kappa: G\ \mbox{is right}\ \kappa\ \mbox{-narrow}\}\ \mbox{and}\ $$$$\mbox{ib}(G)=\omega\cdot\min\{\kappa: G\ \mbox{is left}\ \kappa\ \mbox{-narrow and right}\ \kappa\ \mbox{-narrow}\}.$$ Moreover, we recall the following some definitions.

{\it Character}: $\chi(G)=\omega\cdot\min\{|\mathcal{B}|: \mathcal{B}\ \mbox{is a neighborhood base at the neutral element of}\ G\}$.

{\it Pseudocharacter}: $\psi(G)=\omega\cdot\min\{|\mathcal{U}|: \mathcal{U}\ \mbox{is a family of open sets and}\ \bigcap\mathcal{U}=\{e\}\}$.

{\it Extent}: $e(G)=\omega\cdot\sup\{|S|: S\ \mbox{is a closed discrete subspace of}\ G\}$.

{\it Weakly Lindel\"{o}f degree}: $wl(G)=\omega\cdot\min\{\kappa: \mbox{in each open cover}\ \mathcal{U}\ \mbox{there exists a subfamily}\\ \mathcal{V}\subset\mathcal{U}\ \mbox{with cardinality}\ \kappa\ \mbox{such that}\ \overline{\bigcup\mathcal{V}}=G\}$.

{\it Lindel\"{o}f degree}: $l(G)=\omega\cdot\min\{\kappa: \mbox{in each open cover}\ \mathcal{U}\ \mbox{there exists a subfamily}\\ \mathcal{V}\subset\mathcal{U}\ \mbox{with cardinality}\ \kappa\ \mbox{such that}\ \bigcup\mathcal{V}=G\}$. We say that a space $G$ is $\kappa$-Lindel\"{o}f if $l(G)=\kappa$; in particular, each $\omega$-Lindel\"{o}f space is just a Lindel\"{o}f space.

A semitopological group $G$ is said to be {\it saturated} if, for any non-empty open set $U$, the interior of $U^{-1}$ is non-empty.

The following proposition may have been proven somewhere.

\begin{proposition}\label{p4}
If $G$ is a saturated semitopological group, then $In_{l}(G)=In_{r}(G)$.
\end{proposition}

\begin{proof}
Let $In_{l}(G)=\kappa$. Now we show that $In_{r}(G)\leq\kappa$. Take any open neighborhood $U$ of $e$. Since $G$ is saturated, it follows that $\mbox{int}(U^{-1})\neq\emptyset$. Take any $u\in\mbox{int}(U^{-1})$. Then $u^{-1}\cdot \mbox{int}(U^{-1})$ is an open neighborhood of $e$, hence there exists a subset $A$ with the cardinality of $\kappa$ such that $A\cdot u^{-1}\cdot \mbox{int}(U^{-1})=G$, which shows that $A\cdot u^{-1}\cdot U^{-1}=G$. Thus $U\cdot u\cdot A^{-1}=G$ and $|u\cdot A^{-1}|=|A|=\kappa$. Hence $In_{r}(G)\leq\kappa$. Similarly, one can prove $In_{l}(G)\leq In_{r}(G)$. Therefore, $In_{l}(G)=In_{r}(G)$.
\end{proof}

In \cite[Theorem 3.2]{XST}, the authors proved that $ib(G)\leq e(G)$ for each quasitopological group, and in \cite{RA} the author proved that $ib(G)\leq wl(G)$ for each saturated paratopological group. Therefore, we have the following question by applying Proposition~\ref{p4}.

\begin{question}
If $G$ is a saturated 2-semitopological group, then is $$ib(G)\leq\max\{e(G), wl(G)\}?$$
\end{question}

Moreover, we have the following question.

\begin{question}
If $G$ is a 2-semitopological $T_{1}$ group, then does $nw(G)\leq\chi(G)l(G^{2})$ hold?
\end{question}

Next we discuss the quotient group on $m$-semitopological groups. First, we give a lemma.

\begin{lemma}\label{l2}
Let $G$ be a $T_{1}$ $m$-semitopological group and $F$ be a compact subset with $e\not\in F$, where $m\in\mathbb{N}$. Then it can find an open neighborhood $U$ of $e$ in $G$ such that $e\not\in FU^{m-1}$.
\end{lemma}

\begin{proof}
Since $G$ is $T_{1}$ and $e\not\in F$, we can choose, for each $x\in F$, an open neighborhood $V_{x}$ of $e$ such that $e\not\in xV_{x}$ and $x^{-1}\not\in V_{x}^{m}$. Clearly, the family $\{xV_{x}: x\in F\}$ covers the compact set $F$, hence there exists a finite set $A$ such that $F\subset \bigcup_{a\in A}aV_{a}$. Now put $U=\bigcap_{a\in A}V_{a}$. We claim that $e\not\in FU^{m-1}$. Indeed, for any $f\in F$, there is $b\in A$ such that $f\in bV_{b}$. Since $b^{-1}\not\in V_{b}^{m}$ and $fU^{m-1}\subset bV_{b}U^{m-1}\subset bV_{b}^{m}$, it follows that $e\not\in fU^{m-1}$. Thus $e\not\in FU^{m-1}$.

\end{proof}

\begin{theorem}\label{t2}
Let $G$ be a $T_{1}$ $m$-semitopological group and $H$ a compact closed normal subgroup of $G$, where $m\in\mathbb{N}\setminus\{1\}$. Then $G/H$ is an $(m$-$1)$-semitopological group.
\end{theorem}

\begin{proof}
Clearly, $G/H$ is a $T_{1}$ semitopological group. Take any $g\not\in H$ in $G$. Since $G$ is a $m$-semitopological $T_{1}$ group and $H$ a compact closed normal subgroup of $G$, there exists an neighborhood $U$ of $e$ such that $g\not\in U^{m}$ and $e\not\in g^{-1}HU^{m-1}$ by Lemma~\ref{l2}. We claim that $\pi(g)\not\in (\pi(U))^{m-1}$. Otherwise, $Hg\cap HU^{m-1}\neq\emptyset$, that is, $g\in HU^{m-1}\neq\emptyset$, which shows that $e\in g^{-1}HU^{m-1}$. This is a contradiction. Hence $\pi(g)\not\in (\pi(U))^{m-1}$. Then $G/H$ is an $(m$-$1)$-semitopological group.
\end{proof}

By Theorem~\ref{t2}, we have the following corollary.

\begin{corollary}
Let $G$ be a $T_{1}$ $\infty$-semitopological group and $H$ a compact closed neutral subgroup of $G$. Then $G/H$ is an $\infty$-semitopological group.
\end{corollary}

The following result shows that the cardinality of some 2-semitopological groups is at most $2^{\kappa}$.

\begin{theorem}\label{t6}
If $G$ is a $T_{1}$ 2-semitopological group such that $l(G^{2})\leq \kappa$ and $\psi(G)\leq \kappa$, then $G$ has cardinality at most $2^{\kappa}$.
\end{theorem}

\begin{proof}
Since $G$ is a $T_{1}$ 2-semitopological group, it follows from \cite[Proposition 6 (4)]{RE} that Sym$G$ embeds closed in $G^{2}$, then $l(SymG)\leq \kappa$ by our assumption. Moreover, it is obvious that $SymG$ is $T_{1}$ and $\psi(Sym G)\leq \kappa$. Then Sym$G$ has cardinality at most $2^{\kappa}$ by \cite[Theorem 3.5]{XST}, thus $G$ has cardinality at most $2^{\kappa}$.
\end{proof}

By Theorem~\ref{t6}, we have the following corollary.

\begin{corollary}
If $G$ is a $T_{1}$ 2-semitopological group such that $l(G^{2})\leq \omega$ and $\psi(G)\leq \omega$, then $G$ has cardinality at most $\mathfrak{c}$.
\end{corollary}

Let $\kappa$ be an infinite cardinal. We say that a space $X$ is {\it $\kappa$-cellular} if for each family $\mu$ of $G_{\delta}$-sets of $X$ there exists a subfamily $\lambda\subset \mu$ such that $|\lambda|\leq\kappa$ and $\overline{\bigcup\mu}=\overline{\bigcup\lambda}$.

Finally, we discuss when a 2-semitopological group is a $\kappa$-cellular space. First, we define the class of $\kappa$-$\sum$-spaces and give some lemmas.

Let $\kappa$ be an infinite cardinal. We say that

(1) $X$ is {\it $\kappa$-countably compact} if each open cover of size $\leq\kappa$ has a finite subcover.

(2) $X$ is a $\kappa$-$\sum$-space if there exists a family $\mathscr{P}=\bigcup_{\alpha<\kappa}\mathscr{P}_{\alpha}$  with each $\mathscr{P}_{\alpha}$ being locally finite and the covering of $\mathscr{C}$ by closed $\kappa$-countably compact sets, such that if $C\in \mathscr{C}$ and $C\subset U$ is open, then $C\subset P\subset U$ for some $P\in \mathscr{P}$.

The following proposition is obvious.

\begin{proposition}\label{p5}
A space $X$ is a $\kappa$-$\sum$-space with $e(X)\leq\kappa$ if and only if there exist a family $\mathscr{P}$ with $|\mathscr{P}|\leq\kappa$ and the covering of $\mathscr{C}$ by closed $\kappa$-countably compact sets, such that if $C\in \mathscr{C}$ and $C\subset U$ is open, then $C\subset P\subset U$ for some $P\in \mathscr{P}$.
\end{proposition}

Let $X$ be a space and $\kappa$ an infinite cardinal. We define the following property:

($\mbox{P}_{\kappa}$) Let $\{x_{\alpha}: \alpha<2^{\kappa}\}$ be a subset of $X$ and for each $\alpha<2^{\kappa}$ let $\mathscr{P}_{\alpha}$ be a family of closed subsets of $X$ with a cardinality of at most $\kappa$. Then there is $\beta<2^{\kappa}$ such that the following conditions holds:

($\star$) there exists $y\in \overline{\{x_{\alpha}: \alpha<\beta\}}$ such that if $\eta<\beta$ with $x_{\beta}\in P\in \mathscr{P}_{\eta}$, then $y\in P$.

\begin{lemma}\label{l3}
Let $X$ be a regular $\kappa$-$\Sigma$-space with $e(X)\leq\kappa$, where $\kappa$ is an infinite cardinal. Then ($\mbox{P}_{\kappa}$) holds for $X$.
\end{lemma}

\begin{proof}
By Proposition~\ref{p5}, there exist a family $\mathscr{P}$ with $|\mathscr{P}|\leq\kappa$ and the covering of $\mathscr{C}$ by closed $\kappa$-countably compact sets, such that if $C\in \mathscr{C}$ and $C\subset U$ is open, then $C\subset P\subset U$ for some $P\in \mathscr{P}$. Without loss of generality, we may assume that $\mathscr{P}$ is closed under $<\kappa$ intersections. Let $\{x_{\alpha}: \alpha<2^{\kappa}\}$ be a subset of $X$ and for each $\alpha<2^{\kappa}$ let $\mathscr{F}_{\alpha}$ be a family of closed subsets of $X$ with a cardinality of at most $\kappa$. For each $\mu<2^{\kappa}$, put $$\mathscr{F}_{\mu}^{\ast}=\{\bigcap\mathscr{F}: \mathscr{F}\subset\bigcup_{\alpha<\mu}\mathscr{F}_{\alpha}, |\mathscr{F}|<\kappa\ \mbox{and}\ \bigcap\mathscr{F}\neq\emptyset\}$$ and $$X_{\mu}=\{x_{\alpha}: \alpha<\mu\}.$$

By induction on $\gamma<\kappa$ we construct a family of $\kappa$ ordinals $\{\beta_{\alpha}: \alpha<\kappa\}$ such that for any $0<\alpha<\kappa$ the following two conditions are satisfied:

\smallskip
(i) for any $\alpha<\gamma<\kappa$, we have $\beta_{\alpha}<\beta_{\gamma}$;

\smallskip
(ii) if $x_{\alpha}\in P\cap F$ for $\alpha<2^{\kappa}$, $P\in \mathscr{P}$ and $F\in \mathscr{F}_{\beta_{\eta}}^{\ast}$ for some $\eta<\kappa$, then there exists $y\in \bigcap_{\gamma>\eta}\overline{X_{\beta_{\gamma}}}$ such that $y\in P\cap F$.

Indeed, let $\beta_{0}=\kappa$. Assume that the family $\{\beta_{\eta}: \eta<\alpha\}$ has been constructed, where $\alpha<\kappa$. For $P\in\mathscr{P}$ and $F\in \bigcup_{\eta<\alpha}\mathscr{F}_{\beta_{\eta}}^{\ast}$, let $S(P, F)=\{\nu<2^{\kappa}: x_{\nu}\in P\cap F\}$. If $S(P, F)\neq\emptyset$, then we put $\lambda_{P, F}=\min S(P, F)$. Now we put $$\beta_{\alpha}=\sup \left\{\bigcup_{\eta<\alpha}\beta_{\eta}, \sup\{\lambda_{P, F}: P\in \mathscr{P}, F\in \bigcup_{\eta<\alpha}\mathscr{F}_{\beta_{\eta}}^{\ast}, S(P, F)\neq\emptyset\}\right\}+1.$$
Then the family of $\kappa$ ordinals $\{\beta_{\alpha}: \alpha<\kappa\}$ has been constructed. Put $\beta=\sup\{\beta_{\alpha}: \alpha<\kappa\}$. Now it suffices to check that ($\star$) holds in ($\mbox{P}_{\kappa}$) definition. Clearly, there exists $C\in\mathscr{C}$ such that $x_{\beta}\in C$. We can assume that $$\mathscr{P}_{C}=\{P\in \mathscr{P}: C\subset P\}=\{P_{\alpha, C}: \alpha<\kappa\}$$ and $$\mathscr{F}_{C}=\{F\in \mathscr{F}_{\beta}^{\ast}: x_{\beta}\in F\}=\{F_{\alpha}^{\ast}: \alpha<\kappa\}.$$

Take any $\nu<\kappa$. Let $P_{\nu}=\bigcap_{\alpha\in\nu}P_{\alpha, C}$ and $F_{\nu}=\bigcap_{\alpha\in\nu}F_{\alpha}^{\ast}$, and let $\lambda_{\nu}=\lambda_{P_{\nu}, F_{\nu}}$ and $z_{\nu}=x_{\lambda_{\nu}}$. Clearly, $\mathscr{F}_{\beta}^{\ast}=\bigcap_{\alpha<\kappa}\mathscr{F}_{\beta_{\alpha}}^{\ast}$, hence $F_{\nu}\in \mathscr{F}_{\beta_{\alpha}}^{\ast}$ for some $\alpha<\kappa$. Since $\beta\in \lambda_{P_{\nu}, F_{\nu}}\neq\emptyset$, it follows that $\lambda_{\nu}=\lambda_{P_{\nu}, F_{\nu}}\leq \beta_{\alpha+1}<\beta$, hence $z_{\nu}\in X_{\beta}$.

From the definition of the families $\mathscr{P}$ and $\mathscr{C}$, it follows that $\{z_{v}: v\in\kappa\}$ accumulates to some point $z\in C\cap \bigcap_{\nu<\kappa}F_{\nu}$. Thus, $z\in \overline{X_{\beta}}$. Assume that $F\in \mathscr{F}_{\gamma}$ for $\gamma<\beta$ and $x_{\beta}\in F$. Then $F\in \mathscr{F}_{C}$ and $F=F_{\alpha}^{\ast}\supset F_{\alpha}$ for some $\alpha<\kappa$. Therefore, it follows that $$z\in\bigcap_{\nu<\kappa}F_{\nu}\subset F_{\alpha}\subset F.$$
\end{proof}

\begin{lemma}\label{l4}
Assume that $G$ is a regular quasitopological group, and assume that $G$ satisfies ($\mbox{P}_{\kappa}$) for some infinite cardinal $\kappa$. Then $G$ is a $\kappa$-cellular space.
\end{lemma}

\begin{proof}
Assume that $G$ is not a $\kappa$-cellular space. Then we can find a family $\{A_{\alpha}: \alpha<2^{\kappa}\}$ of non-empty sets of type $G_{\delta}$ such that
$A_{\gamma}\not\subset\overline{\bigcup_{\alpha<\gamma}A_{\alpha}}$ for any $\gamma<2^{\kappa}$. For each $\gamma<2^{\kappa}$, we can pick any $g_{\gamma}\in A_{\gamma}\setminus\overline{\bigcup_{\alpha<\gamma}A_{\alpha}}$, and take a sequence $(U_{\gamma, n})_{n\in\omega}$ of open sets of $G$ such that $g_{\gamma}\in U_{\gamma, n+1}\subset \overline{U_{\gamma, n+1}}\subset U_{\gamma, n}$ for any $n\in\omega$ and $B_{\gamma}=\bigcap_{n\in\omega}U_{\gamma, n}\subset A_{\gamma}$.
For each $\gamma<2^{\kappa}$, put $\mathscr{F}_{\gamma}=\{(G\setminus U_{\beta, n})g_{\alpha}^{-1}g_{\beta}: \alpha, \beta<\gamma, n\in\omega\}$ and $\mathscr{P}_{\gamma}=\mathscr{F}_{\gamma+1}$. Then the condition ($\mbox{P}_{\kappa}$) is satisfied for $G$, it follows that there exists $\delta\in 2^{\kappa}$ and $y\in \overline{\{g_{\alpha}: \alpha<\delta\}}$ such that if $\eta<\delta$, $P\in\mathscr{P}_{\eta}$ and $g_{\delta}\in P$, then $y\in P$. Therefore, $y\in P$ if $g_{\delta}\in P\in \mathscr{P}_{\delta}$. Now, for any $\eta<\delta$, put $y_{\eta}=g_{\delta}y^{-1}g_{\eta}$; we claim that $y_{\eta}\in B_{\eta}$. Suppose not, then there exists $n\in\omega$ such that $y_{\eta}\not\in U_{\eta, n}$. Clearly, we have $y\in g_{\eta}U_{\eta, n+1}^{-1}y\cap g_{\eta}(G\setminus \overline{U_{\eta, n+1}})^{-1}g_{_{\delta}}$. Since $y\in \overline{\{g_{\alpha}: \alpha<\delta\}}$ and $g_{\eta}U_{\eta, n+1}^{-1}y$, $g_{\eta}(G\setminus \overline{U_{\eta, n+1}})^{-1}g_{_{\delta}}$ are open, there exists $\alpha<\delta$ such that $g_{\alpha}\in g_{\eta}U_{\eta, n+1}^{-1}y\cap g_{\eta}(G\setminus \overline{U_{\eta, n+1}})^{-1}g_{_{\delta}}$. Then $g_{\delta}\in (G\setminus U_{\eta, n+1})g_{\eta}^{-1}g_{\alpha}$ and $y\in (U_{\eta, n+1})g_{\eta}^{-1}g_{\alpha}$, which is a contradiction.

Then since $y\in \overline{\{g_{\eta}: \eta<\delta\}}$, it follows that $$g_{\delta}=g_{\delta}y^{-1}y\in \overline{\{g_{\delta}y^{-1}g_{\eta}: \eta<\delta\}}=\overline{\{y_{\eta}: \eta<\delta\}},$$ hence $g_{\delta}\in \overline{\bigcup_{\eta<\delta} B_{\eta}}$. However, it is obvious that $g_{\delta}\not\in \overline{\bigcup_{\eta<\delta} B_{\eta}}$, which is a contradiction. Therefore, $G$ is a $\kappa$-cellular space.

\end{proof}

By Lemmas~\ref{l3} and~\ref{l5}, we have the following lemma.

\begin{lemma}\label{l5}
Let $X$ be a regular quasitopological group, which is a $\kappa$-$\Sigma$-space with $e(G)\leq\kappa$. Then $G$ is a $\kappa$-cellular space.
\end{lemma}

\begin{theorem}\label{t3}
Let $G$ be a regular 2-semitopological group and $G^{2}$ be a $\kappa$-$\Sigma$-space with $e(G)\leq\kappa$. Then $G$ is a $\kappa$-cellular space.
\end{theorem}

\begin{proof}
From \cite[Proposition~6]{RE}, it follows that Sym $G$ is a quasitopological group and embeds closed in $G^{2}$. Then Sym$G$ is a regular a $\kappa$-$\Sigma$-space with $e(G)\leq\kappa$. By Lemma~\ref{l5}, Sym$G$ is a $\kappa$-cellular space. Since $G$ is a continuous image of Sym$G$, it follows that $G$ is a $\kappa$-cellular space.
\end{proof}

By a similar proof of the product of two Lindel\"{o}f $\Sigma$-spaces being Lindel\"{o}f $\Sigma$-space (see \cite{TVV}), we have the following lemma.

\begin{lemma}\label{l6}
Let $X$ be a regular $\kappa$-Lindel\"{o}f $\kappa$-$\Sigma$ space. Then $X^{2}$ is a $\kappa$-Lindel\"{o}f $\kappa$-$\Sigma$ space.
\end{lemma}

By Theorem~\ref{t3} and Lemma~\ref{l6}, we have the following theorem.

\begin{theorem}
Let $G$ be a regular $\kappa$-Lindel\"{o}f $\kappa$-$\Sigma$ 2-semitopological group. Then $G$ is a $\kappa$-cellular space.
\end{theorem}

{\bf Acknowledgements}
We wish to express our sincere thanks to professor T. Banakh who
provided Example~\ref{e0}. Moreover, we wish to thank
the reviewers for careful reading preliminary version of this paper and providing many valuable suggestions.

\smallskip
{\bf Declarations}

\smallskip
{\bf Ethical Approval}

\smallskip
This declaration is ``not applicable''.

\smallskip
{\bf Competing interests}

\smallskip
The authors declare that they have no conflict of interest.

\smallskip
{\bf Availability of data and materials}

\smallskip
This declaration is ``not applicable''.

\end{document}